\documentclass{amsart}
\usepackage{amssymb}
\usepackage{float}
\usepackage{array}

\makeatletter
\newcommand{\thickhline}{%
    \noalign {\ifnum 0=`}\fi \hrule height 1pt
    \futurelet \reserved@a \@xhline
}
\newcolumntype{"}{@{\hskip\tabcolsep\vrule width 1pt\hskip\tabcolsep}}
\makeatother

\newtheorem{theorem}{Theorem}[section]
\newtheorem{lemma}[theorem]{Lemma}

\theoremstyle{definition}

\newtheorem{corollary}[theorem]{Corollary}
\newtheorem{proposition}[theorem]{Proposition}
\newtheorem{example}[theorem]{Example}

\theoremstyle{remark}
\newtheorem{remark}[theorem]{Remark}

\numberwithin{equation}{section}

\newcommand{\al}{\alpha}

\begin{document}

\title{On Pseudo Symmetric Monomial Curves}

\author{Mesut \c{S}ah\.{i}n}
\address{Department of Mathematics,
Hacettepe University, Ankara,  06800 Turkey}
\email{mesut.sahin@hacettepe.edu.tr}
\author{N\.{i}l \c{S}ah\.{i}n}
\address{Department of Industrial Engineering, Bilkent University, Ankara, 06800 Turkey}
\email{nilsahin@bilkent.edu.tr}

\thanks{The authors were supported by the project 114F094 under the program 1001 of the
Scientific and Technological Research Council of Turkey.}

\subjclass[2000]{Primary 13H10, 14H20; Secondary 13P10}
\keywords{Hilbert function, tangent cone, monomial
curve, numerical semigroup, free resolution, Rossi's
conjecture,indispensable binomial}

\date{\today}

\commby{}

\dedicatory{}

\begin{abstract}
We study monomial curves, toric ideals and monomial algebras associated to $4$-generated pseudo symmetric numerical semigroups. Namely, we determine indispensable binomials of these toric ideals, give a characterization for these monomial algebras to have strongly indispensable minimal graded free resolutions. We also characterize when the tangent cones of these monomial curves at the origin are Cohen-Macaulay. \end{abstract}

\maketitle

\section{Introduction}
Characterising numerical functions that may be Hilbert functions of one dimensional Cohen-Macaulay local rings is a hard and still open question of local algebra, see \cite{rossi}. A necesseary condition for the characterization is provided by Sally's conjecture that \textit{the Hilbert function of a one dimensional Cohen-Macaulay local ring with small enough embeddding dimension is non-decreasing.} This conjecture is obvious in embedding dimension one, proved in embedding dimensions two by Matlis \cite{matlis} and three by Elias \cite{elias}. For embedding dimension $4$, Gupta and Roberts gave counterexamples in \cite{GuRo}, and for each embedding dimension greater than $4$, Orecchia gave counterexamples in \cite{O}. Local rings of monomial curves provided many affirmative answers, see e.g. \cite{CJZ,DMM,JZ,PT} and references therein. On the other hand, counterexamples were given only in affine $10$-space by Herzog and Waldi in \cite{HW} and in affine $12$-space by Eakin and Sathaye in \cite{ES}, and most recently, Oneto et al. \cite{OST,OT} announced some methods for producing Gorenstein monomial curves whose tangent cones have decreasing Hilbert functions. However, the problem is still open for monomial curves in $n$-space, where $3<n<10$. As the original conjecture predicts that the embedding dimension $n$ should be small and $4$ is the first case, it is natural to focus on monomial curves in $4$-space. Arslan and Mete gave an affirmative answer to the conjecture for local rings corresponding to $4$-generated symmetric semigroups in \cite{pf} under a numerical condition by proving that the tangent cone is Cohen-Macaulay. Taking the novel aproach to use indispensable binomials in the toric ideal, Arslan et al. refined in \cite{AKN} this by characterising Cohen-Macaulayness of the tangent cone completely. As symmetric and pseudo symmetric semigroups are maximal with respect to inclusion with fixed genus, see \cite{BDF}, the second interesting case is the class of $4$-generated pseudo symmetric semigroups which is the content of the present paper. We give characterizations under which the tangent cone is Cohen-Macaulay. This reveals how nice the singularity at the origin is and verifies Sally's conjecture by \cite{garcia}. It also reduces 
the computation of the Hilbert function to that of its Artinian reduction which have only a finite number of nonzero values, see \cite{Sally}. Our criteria for the Cohen-Macaulayness is in terms of the $5$ integers determining the semigroup, so they can be used in principal to construct counterexamples if there are any.
In order to get these conditions we use indispensable binomials in the toric ideal. Motivated originally from its applications in Algebraic Statistics many authors have studied the concept of indispensability, see e.g. \cite{tak} and \cite{CKT,goj,koj} and later strong indispensability, see \cite{barucci,haraCM,haraJA}. In order to state our results more precisely we introduce some notations.

Let $n_1,\dots, n_4$ be positive integers with $\gcd(n_1,\dots,n_4)=1$. Then the numerical semigroup $S=\langle n_1,\dots,n_4 \rangle$ is defined to be the set $\{u_1n_1+\cdots+u_4n_4\ |\ u_i\in \mathbb{N}\}$. Let $K$ be a field and $K[S]=K[t^{n_1},\ldots,t^{n_4}]$   be the semigroup ring of $S$, then $K[S]\simeq A/I_S$ where, $A=K[X_1,\ldots,X_4]$ and $I_S$ is the kernel of the surjection
$A \stackrel{\phi_0}{\longrightarrow} K[S]$, where $X_i\mapsto t^{n_i}$. 

Pseudo Frobenious numbers of $S$ are defined to be the elements of the set $PF(S)=\{n\in \mathbb{Z}-S\ |\ n+s \in S \hbox{ for all } s\in S-\{ 0\} \}$. The largest pseudo Frobenious number not belonging to $S$ is called the Frobenious number and is denoted by $g(S)$. S is called pseudo symmetric if $PF(S)=\{g(S)/2,g(S)\}$, see \cite[Chapter 3]{RG} or \cite{BDF}. By \cite[Theorem 6.5, Theorem 6.4]{komeda}, the semigroup 
$S$ is pseudo symmetric if and only if there are integers $\alpha_i>1$,
$1\le i\le4$, and $\alpha_{21}>0$, with $\alpha_{21}<\alpha_1$,
such that $n_1=\alpha_2\alpha_3(\alpha_4-1)+1$,
$n_2=\alpha_{21}\alpha_3\alpha_4+(\alpha_1-\alpha_{21}-1)(\alpha_3-1)+\alpha_3$,
$n_3=\alpha_1\alpha_4+(\alpha_1-\alpha_{21}-1)(\alpha_2-1)(\alpha_4-1)-\alpha_4+1$,
$n_4=\alpha_1\alpha_2(\alpha_3-1)+\alpha_{21}(\alpha_2-1)+\alpha_2$. 

From now on, $S$ is assumed to be a pseudo symmetric numerical semigroup. Then, by \cite{komeda},
$K[S]=A/(f_1,f_2,f_3,f_4,f_5)$, where
\begin{eqnarray*} f_1&=&X_1^{\alpha_1}-X_3X_4^{\alpha_4-1},  \quad \quad
f_2=X_2^{\alpha_2}-X_1^{\alpha_{21}}X_4, \quad
f_3=X_3^{\alpha_3}-X_1^{\alpha_1-\alpha_{21}-1}X_2,\\
f_4&=&X_4^{\alpha_4}-X_1X_2^{\alpha_2-1}X_3^{\alpha_3-1}, \quad 
f_5=X_3^{\alpha_3-1}X_1^{\alpha_{21}+1}-X_2X_4^{\alpha_4-1}.
\end{eqnarray*}

In section two, we determine indispensable binomials of $I_S$ and prove that $K[S]$ has a strongly indispensable minimal $S$-graded free resolution if and only if $\al_4>2$ and $\al_1-\al_{21}>2$, see Theorem \ref{main}, filling a missing case in \cite{barucci}.

  In section three, we consider the affine curve $C_S$ with parametrization $$X_1=t^{n_1},\ X_2=t^{n_2},\ X_3=t^{n_3},\ X_4=t^{n_4} $$ corresponding to $S$. Recall  that the local ring corresponding to
the monomial curve $C_S$ is $R_S=K[[t^{n_1},\dots,t^{n_4}]]$ and its
Hilbert function is defined as the Hilbert function of its associated graded ring,
$gr_m(K[[t^{n_1},\dots,t^{n_4}]])$, which is isomorphic to the ring
$K[S]/{I_S}_{*}$. Here, ${I_S}_{*}$ is the defining ideal of
the tangent cone of $C_S$ at the origin and is generated by the homogeneous summands $f_*$ of the elements $f \in I_S$. We characterize when the tangent cone of $C_S$ is Cohen-Macaulay in terms of the defining integers $\al_i$ and $\al_{21}$. As a byproduct of our proofs, we provide explicit generating sets for Cohen-Macaulay tangent cones. 

\section{indispensability} In this section, we determine the indispensable binomials in $I_S$ and characterize the conditions under which $K[S]$ has a strongly indispensable minimal $S$-graded free resolution. First, recall some notions from \cite{CKT}. The $S$-degree of a monomial is defined to be $\deg_S(X_1^{u_1}X_2^{u_2}X_3^{u_3}X_4^{u_4})=\sum_{i=1}^{4}u_in_i \in S$. Let $V( d)$ be the set of monomials of $S$-degree $d$. Denote by $G(d)$ the graph with vertices the elements of $V( d)$ and edges $\{m,n\} \subset V( d)$ such that the binomial $m-n$ is generated by binomials in $I_S$ of $S$-degree strictly smaller than $d$. In particular, when $\gcd(m,n) \neq 1$, $\{m,n\}$ is an edge of $G( d)$. $d \in S$ is called a Betti $S$-degree if there is a minimal generator of $I_S$ of $S$-degree $d$ and $\beta_d$ is the number of times $d$ occurs as a Betti $S$-degree. Both the set $B_S$ of Betti $S$-degrees and $\beta_d$ are invariants of $I_S$. $S$-degrees of binomials in $I_S$ which are not comparable with respect to $<_S$ constitute a subset denoted by $M_S$ whose elements are called minimal binomial $S$-degrees, where $s_1<_S s_2$ if $s_2-s_1 \in S$. In general,  $M_S\subseteq B_S$. By Komeda's result, $B_S=\{d_1,d_2,d_3,d_4,d_5\}$ if $d_i$'s are all distinct, where $d_i$ is the $S$-degree of $f_i$, for $i=1,\dots,5$. A binomial is called indispensable if it appears in every minimal generating set of $I_S$. The following useful observation to detect indispensable binomials is not explicitly stated in \cite{CKT}.
\begin{lemma}\label{kayip} A binomial of $S$-degree $d$ is indispensable if and only if $\beta_d=1$ and $d\in M_S$. 
\end{lemma}
\begin{proof} A binomial of $S$-degree $d$ is indispensable if and only if $G(d)$ has two connected components which are singletons, by \cite[Corollary 2.10]{CKT}. From the paragraph just after \cite[Corollary 2.8]{CKT}, the condition that $G(d)$ has two connected components is equivalent to $\beta_d=1$. Finally, \cite[Proposition 2.4]{CKT} completes the proof, since the connected components of $G(d)$ are singletons if and only if $d\in M_S$. 
\end{proof}

We use the following many times in the sequel.
\begin{lemma}\label{lemmaS}
If $0< v_k <\alpha_k$ and $0< v_l<\alpha_l$, for $k \neq l \in\{1,2,3,4\}$, then $v_kn_k-v_ln_l \notin S$. 
\end{lemma}
\begin{proof}
Assume to the contrary that $v_kn_k-v_ln_l \in S$. Then 
$$v_kn_k-v_ln_l =\sum\limits_{i=1}^{4} u_in_i=u_1n_1+u_2n_2+u_3n_3+u_4n_4$$ 
for some non-negative $u_k$'s. 

Hence, $(v_k-u_k)n_k=(v_l+u_l)n_l+u_sn_s+u_rn_r \in \langle n_l,n_s,n_r \rangle$. If $v_k-u_k<0$ then $(v_k-u_k)n_k \in S \cap (-S)$ but this is a contradiction as $S \cap (-S)=\{0\}$. If $v_k-u_k=0$, then $(v_l+u_l)n_l+u_sn_s+u_rn_r=0$ and this is impossible as $v_l$ is positive. That is, $v_k-u_k >0$. This contradicts with the fact that $\alpha_i$ is the smallest positive number with this property as $0<v_i-u_i\leq v_i<\alpha_i$. 
\end{proof}

Now, we determine the minimal binomial $S$-degrees. 
\begin{proposition}\label{minimal}
$M_S=\{d_1,d_2,d_3,d_4,d_5\}$ if $\alpha_1-\alpha_{21} > 2$ and $M_S=\{d_1,d_2,d_3,d_5\}$ if $\alpha_1-\alpha_{21} =2$.
\end{proposition}
\begin{proof} Notice first that $d_1=\alpha_1n_1=n_3+(\alpha_4-1)n_4$,\\
$d_2=\alpha_2n_2=\alpha_{21}n_1+n_4$,\\
$d_3=\alpha_3n_3=(\alpha_1-\alpha_{21}-1)n_1+n_2$,\\
$d_4=\alpha_4n_4=n_1+(\alpha_2-1)n_2+(\alpha_3-1)n_3$,\\
$d_5=(\alpha_{21}+1)n_1+(\alpha_3-1)n_3=n_2+(\alpha_4-1)n_4.$\\
Thus, we observe that\\
$d_1-d_2=(\alpha_1-\alpha_{21})n_1-n_4$\\
$d_1-d_3=(\alpha_{21}+1)n_1-n_2$\\
$d_1-d_4=n_3-n_4$\\
$d_1-d_5=(\alpha_1-\alpha_{21}-1)n_1-(\alpha_3-1)n_3$\\
$d_2-d_3=(\alpha_2-1)n_2-(\alpha_1-\alpha_{21}-1)n_1$\\
$d_2-d_4=n_3-(\alpha_1-\alpha_{21})n_1$\\
$d_2-d_5=(\alpha_2-1)n_2-(\alpha_4-1)n_4$\\
$d_3-d_4=n_3-n_1-(\alpha_2-1)n_2$\\
$d_3-d_5=n_3-(\alpha_{21}+1)n_1$\\
$d_4-d_5=(\alpha_2-1)n_2-\alpha_{21}n_1$.\\
Then, $d_i-d_j=v_kn_k-u_ln_l$ for some $k \neq l \in\{1,2,3,4\}$ with $0< v_k <\alpha_k$ and $0< v_l<\alpha_l$  except for $d_3-d_4$ and $d_4-d_3$. Hence, we can say $d_i-d_j \notin S$ from Lemma \ref{lemmaS} for all $i,j$ except $3$ and $4$.

Assume $d_3-d_4 \in S$. Then $n_3-n_1-(\alpha_2-1)n_2=u_1n_1+u_2n_2+u_3n_3+u_4n_4$ for some non-negative $u_i$'s. So, $(1-u_3)n_3=(1+u_1)n_1+(\alpha_2-1+u_2)n_2+u_4n_4 >0$. This contradicts to $\alpha_3$ being the minimal number with the property $\alpha_3n_3 \in \langle n_1,n_2,n_4 \rangle$, as $0<1-u_3<\alpha_3$. Hence $d_3-d_4$ can not be in $S$.

There are two possibilities for $d_4-d_3$. If $\alpha_1-\alpha_{21} =2$, then we have $d_4-d_3=(\alpha_2-2)n_2+(\alpha_3-1)n_3-(\alpha_1-\alpha_{21}-2)n_1=(\alpha_2-2)n_2+(\alpha_3-1)n_3 \in S$. 

If $\alpha_1-\alpha_{21} > 2$, we show that $d_4-d_3 \notin S$. Assume contrary that $d_4-d_3=n_1+(\alpha_2-1)n_2-n_3=u_1n_1+u_2n_2+u_3n_3+u_4n_4$. Then, $(\alpha_2-1-u_2)n_2=(u_1-1)n_1+(u_3+1)n_3+u_4n_4$. If $u_1 >0$, then $0<\alpha_2-1-u_2<\alpha_2$, since $u_3+1>0$. But this contradicts to the minimality of $\alpha_2$.  Hence $u_1=0$ and $n_1+(\alpha_2-1-u_2)n_2=(u_3+1)n_3+u_4n_4$ with $\alpha_2-1-u_2>0$. ( If $\alpha_2-1-u_2\leq 0$, then $n_1=(u_2+1-\alpha_2)n_2+(u_3+1)n_3+u_4n_4$ and this implies $n_1 \in \langle n_2,n_3,n_4 \rangle$  which can not happen). Then if $u_4=0$, we have $(u_3+1)n_3=n_1+(\alpha_2-1-u_2)n_2$. As $u_3+1<\alpha_3$  gives a contradiction with the minimality of $\alpha_3$, we assume $u_3+1=\alpha \geq \alpha_3$. Then $\alpha_3n_3+(\alpha-\alpha_3)n_3=n_1+(\alpha_2-1-u_2)n_2$ $\Rightarrow (\alpha_1-\alpha_{21}-1)n_1+n_2+(\alpha-\alpha_3)n_3=n_1+(\alpha_2-1-u_2)n_2$ $\Rightarrow (\alpha_1-\alpha_{21}-2)n_1+(\alpha-\alpha_3)n_3=(\alpha_2-2-u_2)n_2$ $\Rightarrow 0< \alpha_2-2-u_2 < \alpha_2$ and this gives a contradiction with the minimality of $\alpha_2$. On the other hand, if $u_4>0$, then $n_1+\alpha_2n_2=(1+u_2)n_2+(u_3+1)n_3+u_4n_4$, and as $\alpha_2n_2=1+\alpha_{21}n_1+n_4$, we have $(1+\alpha_{21})n_1=(1+u_2)n_2+(u_3+1)n_3+(u_4-1)n_4$. As $0<1+\alpha_{21}<\alpha_1$,  this contradicts with the minimality of $\alpha_1$. Hence, $d_4-d_3$ can not be an element of $S$.
\end{proof}

As a consequence, we determine the indispensable binomials in $I_S$. Part of this result is remarked at the end of \cite{koj}.
\begin{corollary}\label{indispensable} Indispensable binomials of $I_S$ are $\{f_1,f_2,f_3,f_4,f_5\}$ if $\alpha_1-\alpha_{21} > 2$ and are $\{f_1,f_2,f_3,f_5\}$ if $\alpha_1-\alpha_{21} =2$.
\end{corollary}

\begin{proof} This follows from Lemma \ref{kayip} and Proposition \ref{minimal}, since $\beta_{d_i}=1$, for all $i=1,\dots,5$.   
\end{proof}

A minimal graded free resolution of $K[S]$ is given in \cite[Theorem 6]{barucci} as follows:

\begin{theorem} \label{res}If $S$ is a $4$-generated pseudosymmetric semigroup, then the following is a minimal graded free $A$-resolution of $K[S]$:
$$({\bf F},\phi): 0\longrightarrow \bigoplus_{j=1}^{2} A[-c_j] \stackrel{\phi_3}{\longrightarrow}\bigoplus_{j=1}^{6} A[-b_j]\stackrel{\phi_2}{\longrightarrow}\bigoplus_{j=1}^{5} A[-d_j]\stackrel{\phi_1}{\longrightarrow}A\longrightarrow 0$$
where $\phi_1=(f_1,f_2,f_3,f_4,f_5),$

$$\phi_2=\left( \begin{array}{cccccc}X_2&0&X_3^{\alpha_3-1}&0&X_4&0\\
0&f_3&0&X_1X_3^{\alpha_3-1}&X_1^{\alpha_1-\alpha_{21}}&X_4^{\alpha_4-1}\\
X_1^{\alpha_{21}+1}&-f_2&X_4^{\alpha_4-1}&0&X_1X_2^{\alpha_2-1}&0\\
0&0&0&X_2&X_3&X_1^{\alpha_{21}}\\
-X_3&0&-X_1^{\alpha_1-\alpha_{21}-1}&X_4&0&X_2^{\alpha_2-1}\\
\end{array}\right),
$$

\bigskip

\noindent and $\phi_3=\left( \begin{array}{cccccc}X_4&-X_1&0&X_3&-X_2&0\\
-X_2^{\alpha_2-1}X_3^{\alpha_3-1}&X_4^{\alpha_4-1}&f_2&-X_1^{\alpha_1-1}&X_1^{\alpha_{21}}X_3^{\alpha_3-1}&-f_3
\end{array}\right)^T.$
\end{theorem}
The numbers $b_j$ and $c_j$ above can be obtained from the maps $\phi_2$ and  $\phi_3$ as in \cite[Corollary 16]{barucci}. For instance, the $S$-degrees of the non-zero entries in the first column of $\phi_2$ gives us 
$b_1=d_1+n_2=d_3+(\alpha_{21}+1)n_1=d_5+n_3$. Similarly we get:\\
$b_2=d_2+d_3$\\
$b_3=d_1+(\alpha_3-1)n_3=d_3+(\alpha_4-1)n_4=d_5+(\alpha_1-\alpha_{21}-1)n_1$\\
$b_4=d_4+n_2=d_2+n_1+(\alpha_3-1)n_3=d_5+n_4$\\
$b_5=d_1+n_4=d_2+(\alpha_1-\alpha_{21})n_1=d_3+n_1+(\alpha_2-1)n_2=d_4+n_3$\\
$b_6=d_2+(\alpha_4-1)n_4=d_4+\alpha_{21}n_1=d_5+(\alpha_2-1)n_2$\\
and\\
$c_1=b_1+n_4=b_2+n_1=b_4+n_3=b_5+n_2$\\
$c_2=b_1+(\alpha_2-1)n_2+(\alpha_3-1)n_3$\\
\phantom{$c_2 $ }$= b_2+(\alpha_4-1)n_4$\\
\phantom{$c_2 $ }$= b_3+d_2=b_3+\alpha_2n_2=b_3+\alpha_{21}n_1+n_4$\\
\phantom{$c_2 $ }$= b_4+(\alpha_1-1)n_1$\\
\phantom{$c_2 $ }$= b_5+\alpha_{21}n_1+(\alpha_3-1)n_3$\\
\phantom{$c_2 $ }$= b_6+d_3=b_6+\alpha_3n_3=b_6+(\alpha_1-\alpha_{21}-1)n_1+n_2$.\\

Note that the resolution $({\bf F},\phi)$ is called strongly indispensable if for any graded minimal resolution $({\bf G},\theta)$, we have an injective complex map
$i\colon({\bf F},\phi)\longrightarrow({\bf G},\theta)$. We finish this section with its main result to characterize when $K[S]$ has a strongly indispensable minimal graded free resolution.
\begin{theorem} \label{main}
Let $S$ be a 4-generated pseudo-symmetric semigroup. Then $K[S]$ has a strongly indispensable minimal graded free resolution if and only if $\alpha_4>2$ and $\alpha_1-\alpha_{21}>2$.
\end{theorem}
\begin{proof}
According to \cite[Proposion 29]{barucci}, $K[S]$ has a strongly indispensable minimal graded free resolution if and only if the differences $d_i-d_j$ and $b_i-b_j$ do not belong to $S$, for any $i$ and $j$. Indeed, $d_i-d_j \notin S$ if and only if $\alpha_1-\alpha_{21}>2$ from the proof of Proposition \ref{minimal}. For the other differences, we use the identities in $c_1$ and $c_2$.
As a result, from $c_1$, we get the differences \\
$b_1-b_2=n_1-n_4$, \\
$b_1-b_4=n_3-n_4$, \\
$b_1-b_5=n_2-n_4$, \\
$b_2-b_4=n_3-n_1$, \\
$b_2-b_5=n_2-n_1$, \\
$b_4-b_5=n_2-n_3.$ 

Similarly, from $c_2$, we get the differences \\
$b_1-b_3=n_2-(\alpha_3-1)n_3$, \\
$b_1-b_6=n_3-(\alpha_2-1)n_2$, \\
$b_3-b_4=(\alpha_1-\alpha_{21}-1)n_1-n_4$, \\
$b_3-b_5=(\alpha_3-1)n_3-n_4$, \\
$b_3-b_6=(\alpha_1-\alpha_{21}-1)n_1-(\alpha_2-1)n_2$\\
$b_4-b_6=n_2-\alpha_{21}n_1$, \\
$b_5-b_6=n_3-\alpha_{21}n_1$, 
. 

Observe that $b_i-b_j=v_kn_k-v_ln_l$ for any $i<j$ and for some $k \neq l \in\{1,2,3,4\}$ with $0<v_k<\alpha_k$ and $0<v_l<\alpha_l$. By Lemma \ref{lemmaS}, we have $\mp(b_i-b_j) \notin S$, for any $i<j$, except for $i=2$ and $j=3,6$. 

Furthermore, $b_2-b_3=\alpha_2n_2-(\alpha_4-1)n_4=\alpha_{21}n_1-(\alpha_4-2)n_4$. Again by Lemma \ref{lemmaS},  we have $\mp(b_2-b_3) \notin S$ when $\alpha_4>2$. On the other hand, if $\alpha_4=2$, then $b_2-b_3=\alpha_{21}n_1 \in S$. 

Finally, $b_2-b_6=\alpha_3n_3-(\alpha_4-1)n_4$. Using the identity $$d_5=(\alpha_{21}+1)n_1+(\alpha_3-1)n_3=n_2+(\alpha_4-1)n_4,$$ we obtain 
$b_2-b_6=n_2+n_3-(\alpha_{21}+1)n_1.$ If $b_2-b_6\in S$, then there are non-negative $u_i$ such that
$$n_2+n_3-(\alpha_{21}+1)n_1=b_2-b_6=u_1n_1+u_2n_2+u_3n_3+u_4n_4.$$ Then $(1-u_2)n_2+(1-u_3)n_3=(\alpha_{21}+1+u_1)n_1+u_4n_4>0$. It follows that $u_2=u_3=0$. Thus, $n_2+n_3=(\alpha_{21}+1+u_1)n_1+u_4n_4$. 

If $u_4=0$ then $\alpha n_1 \in \langle n_2,n_3\rangle$ with $\alpha<\alpha_1$ because if $\alpha\geq \alpha_1$, then\\ 
$n_2+n_3=(\alpha-\alpha_1)n_1+\alpha_1n_1=(\alpha-\alpha_1)n_1+n_3+(\alpha_4-1)n_4.$ 
This leads to a contradiction as $n_2=(\alpha-\alpha_1)n_1+n_4(\alpha_4-1) \in \langle n_1,n_4 \rangle$. So $u_4>0$ in which case,
$n_2+n_3=\alpha_{21}n_1+(1+u_1)n_1+n_4+(u_4-1)n_4=(1+u_1)n_1+\alpha_2n_2+(u_4-1)n_4$
$\Rightarrow n_3=(u_1+1)n_1+(\alpha_2-1)n_2+(u_4-1)n_4 \in \langle n_1,n_2,n_4 \rangle$, another contradiction. Hence, $b_2-b_6 \notin S$.

If $b_6-b_2=(\alpha_{21}+1)n_1-n_2-n_3=u_1n_1+u_2n_2+u_3n_3+u_4n_4$, for some non-negative $u_i$, then 
$(\alpha_{21}+1-u_1)n_1=(u_2+1)n_2+(u_3+1)n_3+u_4n_4>0$. Then $0<\alpha_{21}+1-u_1<\alpha_1$, a contradiction with the minimality of $\alpha_1$. Hence, $b_6-b_2$ can not be an element of $S$ either, completing the proof.
\end{proof}

\section{Cohen-Macaulayness of the tangent cone}
In this section, we give conditions for the Cohen-Macaulayness of the tangent cone. For some recent and past activity about the tangent cone of $C_S$, see \cite{arslan,AKN,CZ,HS,shen,shibuta}.

Recall that for an ideal $I$ with a fixed monomial ordering \textquoteleft$< $\textquoteright, a finite set $G \subset I$ is called a standard basis of $I$ if the leading monomials of the elements of $G$ generate the leading ideal of $I$ that is, if for any $f \in I-\{0\}$, there exits $g \in G $ such that ${\rm LM}(g)$ divides ${\rm LM}(f)$. Note that a standard basis is also a basis for the ideal and when the ordering  \textquoteleft$< $\textquoteright\  is global, standard basis is actually a Gr\"{o}bner basis, \cite{greuel-pfister}.

\begin{remark} Depending on the ordering among $n_1, n_2, n_3$ and $n_4$ there are $24$ possible cases.  We illustrate in Table $1$ that there are pseudo symmetric monomial curves with Cohen-Macaulay tangent cones in all of these cases. We will determine standard bases and characterize Cohen-Macaulayness completely in the first $12$ cases in terms of the defining integers. For the remaining $12$ cases, finding a general form for the standard basis is not possible, and instead of giving a characterization as in \cite{AKN}, we give some partial results involving the defining integers $\alpha_i$ and $\alpha_{21}$.
\end{remark}
\begin{table}[H]
\caption{cases
\label{fig:cases}}
\begin{center}
\begin{tabular}{|c"c|c|c|c|c"c|c|c|c|}
  \hline
 & $\alpha_{21}$&$\alpha_1$& $\alpha_2$ & $\alpha_3$& $\alpha_4$&$n_1$&$n_2$&$n_3$&$n_4$\\
  \thickhline
  $n_1<n_2<n_3<n_4$&2 &5 &3 &2 &2 &7 &12 & 13 & 22 \\
  \hline
   $n_1<n_2<n_4<n_3$&2&5&3&2&4&19&20&29&22\\
  \hline
  $n_1<n_3<n_2<n_4$&3&5&4&2&3&17&21&19&33\\
  \hline
  $n_1<n_3<n_4<n_2$&3&6&3&3&5&37&52&42&45\\
  \hline
  $n_1<n_4<n_2<n_3$&3&6&3&2&4&19&28&33&27 \\
  \hline
  $n_1<n_4<n_3<n_2$&3&8&3&4&6&61&88&83&81\\
  \thickhline
   
   $n_2<n_1<n_3<n_4$&2&6&6&3&5&73&39&86&88\\
  \hline
   $n_2<n_1<n_4<n_3$&2&5&4&2&4&25&20&35&30\\
  \hline
   $n_2<n_3<n_1<n_4$&2&4&4&2&4&25&19&22&26\\
  \hline
  $n_2<n_3<n_4<n_1$&3&5&6&2&6&61&39&50&51\\
  \hline
  $n_2<n_4<n_1<n_3$&2&5&4&2&5&33&24&45&30\\
  \hline
   $n_2<n_4<n_3<n_1$&2&4&4&2&5&33&23&28&26\\
  \thickhline
   
   $n_3<n_1<n_2<n_4$&1&3&2&3&3&13&14&9&15\\
  \hline
   $n_3<n_1<n_4<n_2$&3&6&3&4&6&61&82&51&63\\
  \hline
   $n_3<n_2<n_1<n_4$&2&4&4&5&4&61&49&22&74\\
  \hline
   $n_3<n_2<n_4<n_1$&2&4&5&4&5&81&59&28&74\\
  \hline
  $n_3<n_4<n_1<n_2$&2&4&2&4&6&41&55&24&28\\
  \hline
  $n_3<n_4<n_2<n_1$&2&4&3&4&5&49&47&24&43\\
  \thickhline
  
 $n_4<n_1<n_2<n_3$&2&5&2&2&5&17&24&29&14\\
  \hline
  $n_4<n_1<n_3<n_2$&2&4&2&2&4&13&19&16&12\\
  \hline
  $n_4<n_2<n_1<n_3$&1&4&2&2&4&13&12&19&11\\
  \hline
  $n_4<n_2<n_3<n_1$&1&3&2&2&4&13&11&12&9\\
  \hline
  $n_4<n_3<n_1<n_2$&2&5&2&4&6&41&58&35&34\\
  \hline
  $n_4<n_3<n_2<n_1$&1&4&2&4&6&41&34&29&27\\
  \thickhline
\end{tabular}
\end{center}
\end{table}

\subsection{Cohen-Macaulayness of the tangent cone when $n_1$ is smallest}

 In this section, we assume that $n_1$ is the smallest number in $\{n_1,n_2,n_3,n_4\}$. Using the indispensable binomials of $I_S$, we characterize the Cohen-Macaulayness of the tangent cone of $C_S$. First, we get the necessary conditions.

\begin{lemma} \label{necessary1}
If the tangent cone of the monomial curve $C_S$ is Cohen-Macaulay, then the following must hold
\begin{enumerate}
\item[(C1.1)]  $\alpha_2\leq \alpha_{21}+1$, 
\item[(C1.2)]  $\alpha_{21}+\alpha_3\leq \alpha_1$,
\item[(C1.3)]  $\alpha_4\leq \alpha_2+\alpha_3-1$.
\end{enumerate} 
\end{lemma}
\begin{proof}
Corollary \ref{indispensable} implies that $f_2$ and $f_3$ are indispensable binomials of $I_S$, which means that they appear in every standard basis. To prove C$(1.1)$, assume contrary that $\alpha_2 > \alpha_{21}+1$. Then, ${\rm LM}(f_2)=X_1^{\alpha_{21}}X_4$ is divisible by $X_1$. This leads to a contradiction as \cite[Lemma 2.7]{AMS} implies that the tangent cone is not Cohen-Macaulay. Similarly, when $\alpha_{21}+\alpha_3> \alpha_1$, ${\rm LM}(f_3)=X_1^{\alpha_1-\alpha_{21}-1}X_2$ is divisible by $X_1$. So, if the tangent cone is Cohen-Macaulay, then C$(1.1)$ and C$(1.2)$ must hold. 

To show the last inequality holds, assume not: $\alpha_4 > \alpha_2+\alpha_3-1$. Then ${\rm LM}(f_4)=X_1X_2^{\alpha_{2}-1}X_3^{\alpha_3-1}$ is divisible by $X_1$. If $\alpha_1 > \alpha_{21}+2$, $f_4$ is indispensable by Corollary \ref{indispensable} again. As before, the tangent cone is not Cohen Macaulay, a contradiction. So, we must have $\alpha_1=\alpha_{21}+2$. In this case, there exists a binomial $g$ in a minimal standard basis of $I_S$ such that  ${\rm LM}(g) \mid {\rm LM}(f_4)$ and $X_1 \nmid {\rm LM}(g)$. Hence ${\rm LM}(g)=X_2^{a}X_3^{b}$ with $0< a\leq \alpha_2-1$ and $0< b\leq \alpha_3-1$ since the case $a=0$ contradicts with the minimality of $d_2$ and the case $b=0$ contradicts with the minimality of $d_3$. By Proposition \ref{minimal} and its proof, $M_S=\{d_1,d_2,d_3,d_5\}$ are the minimal degrees and the only degree that is smaller than $d_4$ is $d_3$. Since ${\rm deg}(g)<d_4$, we must have $d_3<{\rm deg}(g)<d_4$. Hence, ${\rm deg}(g)-d_3=an_2-(\alpha_3-b)n_3 \in S$ with $0<a<\alpha_2$ and $0<\alpha_3-b<\alpha_3$ but this contradicts to Lemma \ref{lemmaS}. So, C$(1.3)$ must hold as well. 
\end{proof}

Before we check if the conditions C$(1.1)$, C$(1.2)$ and C$(1.3)$ are sufficient, we note the following.
\begin{remark}\label{remark2} $\alpha_{1}\geq \alpha_{4}$ holds. Indeed, as $f_1$ is $S$-homogeneous and $n_1<n_4$, we have $(\alpha_4-1)n_4<n_3+(\alpha_4-1)n_4=\alpha_1n_1<\alpha_1n_4$ implying  $\alpha_{1} > \alpha_{4}-1$.
\end{remark}

Next, we compute a standard basis for $I_S$, when C$(1.1)$, C$(1.2)$ and C$(1.3)$ hold.
\begin{lemma}\label{smallest1} If C$(1.1)$, C$(1.2)$ and C$(1.3)$ hold, the set $G=\{ f_1, f_2, f_3, f_4, f_5 \}$ is a minimal standard basis for $I_S$ with respect to a negative degree reverse lexicographical ordering making $X_1$ the smallest variable.
\end{lemma}
\begin{proof}
We will apply standard basis algorithm to the set $G=\{f_1, f_2, f_3, f_4, f_5\}$ with the normal form algorithm NFM\tiny{ORA}\normalsize, see \cite{greuel-pfister} for details. We need to show $NF ({\rm spoly}(f_i,f_j) \vert G)=0$ for any $i\neq j$ with  $1\leq i,j\leq 5$. Observe that the conditions (C1.1) and (C1.3) imply that $\alpha_4 \leq \alpha_{21}+\alpha_3$(*) and hence,
\begin{itemize}
\item ${\rm LM}( f_1)={\rm LM}(X_1^{\alpha_1}-X_3X_4^{\alpha_4-1})= X_3X_4^{\alpha_4-1}$,\ \ \ by Remark \ref{remark2}
\item ${\rm LM}(f_2)={\rm LM}(X_2^{\alpha_2}-X_1^{\alpha_{21}}X_4)=X_2^{\alpha_2}$,\ \ \ by (C1.1).
\item ${\rm LM}(f_3)={\rm LM}(X_3^{\alpha_3}-X_1^{\alpha_1-\alpha_{21}-1}X_2)= X_3^{\alpha_3}$,\ \ \ by (C1.2)
\item ${\rm LM}(f_4)={\rm LM}(X_4^{\alpha_4}-X_1X_2^{\alpha_2-1}X_3^{\alpha_3-1})=X_4^{\alpha_4}$,\ \ \ by (C1.3)
\item ${\rm LM}(f_5)={\rm LM}(X_1^{\alpha_{21}+1}X_3^{\alpha_3-1}-X_2X_4^{\alpha_4-1})= X_2X_4^{\alpha_4-1}$,\ \ \ by (*).
\end{itemize}
Then we conclude the following:
\begin{itemize}
\item $ NF({\rm spoly}(f_i,f_j) \vert G)= 0$ as ${\rm LM}( f_i)$ and ${\rm LM}( f_j)$ are relatively prime, for $(i,j)\in \{(1,2),(2,3),(2,4),(3,4),(3,5)\}$.

\item ${\rm spoly}(f_1,f_3)=X_1^{\alpha_1}X_3^{\alpha_3-1}-X_1^{\alpha_1-\alpha_{21}-1}X_2X_4^{\alpha_4-1}$ and by (*) its leading monomial is $X_1^{\alpha_1-\alpha_{21}-1}X_2X_4^{\alpha_4-1}$, which is divisible only by ${\rm LM}(f_5)$. As ${\rm ecart} (f_5)= {\rm ecart} ({\rm spoly}(f_1,f_3))$ and ${\rm spoly}(f_5, {\rm spoly}(f_1,f_3))=0$, we have
    $$ NF({\rm spoly}(f_1,f_3) \vert G)= 0.$$

\item ${\rm spoly}(f_1,f_4)=X_1^{\alpha_{1}}X_4-X_1X_2^{\alpha_2-1}X_3^{\alpha_{3}}$.
\begin{eqnarray*}
  \alpha_2 &\leq& \alpha_{21}+1\ \  \hbox{from (C1.1). Then,}\\
  \alpha_2+\alpha_3 &\leq & \alpha_3+\alpha_{21}+1\ \ \hbox{then as\ }  \alpha_3 \leq \alpha_1-\alpha_{21} \hbox{ from (C1.2)} \\
  \alpha_2+\alpha_3& \leq& \alpha_1+1.
\end{eqnarray*}
As a result, ${\rm LM}({\rm spoly}(f_1,f_4))=X_1X_2^{\alpha_2-1}X_3^{\alpha_{3}}$. Only ${\rm LM}(f_3)$ divides ${\rm LM}( {\rm spoly}(f_1,f_4))$ and ${\rm ecart} ({\rm spoly}(f_1,f_4)) \geq  {\rm ecart}(f_3)$. Then,\newline
$ {\rm spoly}(f_3,{\rm spoly}(f_1,f_4))=X_1^{\alpha_1}X_4-X_1^{\alpha_1-\alpha_{21}}X_2^{\alpha_2}$. As $\alpha_2 \leq \alpha_{21}+1$ from (C1.1), $\alpha_1-\alpha_{21}+\alpha_2\leq \alpha_1+1$ and hence ${\rm LM}({\rm spoly}(f_3,{\rm spoly}(f_1,f_4)))=X_1^{\alpha_1-\alpha_{21}}X_2^{\alpha_2}$. Among the leading monomials of elements of $G$, only ${\rm LM}(f_2)$ divides this with
${\rm ecart}(f_2)=\alpha_{21}+1-\alpha_2={\rm ecart}({\rm spoly}(f_3,{\rm spoly}(f_1,f_4))$. Then
${\rm spoly}(f_2,{\rm spoly}(f_3,{\rm spoly}(f_1,f_4)))=0$ implying $$ NF({\rm spoly}(f_1,f_4) \vert G)= 0.$$

\item  ${\rm spoly}(f_1,f_5)=X_1^{\alpha_{21}+1}X_3^{\alpha_3}-X_1^{\alpha_{1}}X_2$ with ${\rm LM}({\rm spoly}(f_1,f_5))=X_1^{\alpha_{21}+1}X_3^{\alpha_{3}}$ by (C1.2).  Only ${\rm LM}(f_3)$ divides this. As ${\rm ecart} ({\rm spoly}(f_1,f_5))= \alpha_1-\alpha_{21}+\alpha_{3} ={\rm ecart}(f_3)$ and
    ${\rm spoly}(f_3,{\rm spoly}(f_1,f_5))=0$, $ NF({\rm spoly}(f_1,f_5) \vert G)= 0.$
    
\item  ${\rm spoly}(f_2,f_5)=X_1^{\alpha_{21}+1}X_2^{\alpha_2-1}X_3^{\alpha_3-1}-X_1^{\alpha_{21}}X_4^{\alpha_{4}}$.
As (C1.3) implies $\alpha_{21}+\alpha_4\leq\alpha_{21}+\alpha_2+\alpha_3-1$,  ${\rm LM}({\rm spoly}(f_2,f_5))=X_1^{\alpha_{21}}X_4^{\alpha_{4}}.$  Only ${\rm LM}(f_4)$ divides this. As
 ${\rm ecart} ({\rm spoly}(f_2,f_5))= \alpha_2+\alpha_3-1-\alpha_{4} ={\rm ecart}(f_4)$ and
    ${\rm spoly}(f_4,{\rm spoly}(f_2,f_5))=0$,
    $ NF({\rm spoly}(f_2,f_5) \vert G)= 0.$
 Finally,
 \item ${\rm spoly}(f_4,f_5)=X_1^{\alpha_{21}+1}X_3^{\alpha_3-1}X_4-X_1X_2^{\alpha_2}X_3^{\alpha_3-1}$. Then $\alpha_2 \leq \alpha_{21}+1$ implies $\alpha_2+\alpha_3 \leq \alpha_{21}+1+\alpha_3$ and hence ${\rm LM}({\rm spoly}(f_4,f_5))=X_1X_2^{\alpha_{2}}X_3^{\alpha_3-1}.$  Only ${\rm LM}(f_2)$ divides this. Since ${\rm ecart}({\rm spoly}(f_4,f_5))= \alpha_{21}+1-\alpha_{2} ={\rm ecart}(f_2)$ and  ${\rm spoly}(f_2,{\rm spoly}(f_4,f_5))=0$, $ NF({\rm spoly}(f_4,f_5) \vert G)= 0.$
\end{itemize}
It is not hard to see that this standard basis is minimal, so we are done.
\end{proof}
We are now ready to give the complete characterization of the Cohen-Macaulayness of the tangent cone.
\begin{theorem}
Suppose $n_1$ is the smallest number in $\{n_1,n_2,n_3,n_4\}$. The tangent cone of $C_S$ is Cohen-Macaulay if and only if 
\begin{enumerate}
\item[(C1.1)]  $\alpha_2\leq \alpha_{21}+1$, 
\item[(C1.2)]  $\alpha_{21}+\alpha_3\leq \alpha_1$,
\item[(C1.3)]  $\alpha_4\leq \alpha_2+\alpha_3-1$.
\end{enumerate} 
\end{theorem}
\begin{proof} If the tangent cone of $C_S$ is Cohen-Macaulay, then C$(1.1)$, C$(1.2)$ and C$(1.3)$ hold, by Lemma \ref{necessary1}. If C$(1.1)$, C$(1.2)$ and C$(1.3)$ hold, then from Lemma \ref{smallest1}, a minimal standard basis for $I_S$ is $G=\{ f_1, f_2, f_3, f_4, f_5 \}$ and $X_1 \nmid {\rm LM}(f_i)$ for $i=1, 2, 3, 4, 5$. Thus, it follows from \cite[Lemma 2.7]{AMS} that the tangent cone is Cohen-Macaulay.
\end{proof}

\subsection{Cohen Macaulayness of the tangent cone when $n_2$ is smallest} In this section, we deal with the Cohen Macaulayness of the tangent cone when $n_2$ is the smallest number in $\{n_1,n_2,n_3,n_4\}$. As before, we get the necessary conditions first.
\begin{lemma} \label{necessary2} Suppose $n_2$ is the smallest number in $\{n_1,n_2,n_3,n_4\}$. If the tangent cone of the monomial curve $C_S$ is Cohen-Macaulay, then the following must hold
\begin{enumerate}
 \item[(C2.1)]  $\alpha_{21}+\alpha_3\leq \alpha_1$, 
 \item[(C2.2)] $\alpha_{21}+\alpha_3\leq \alpha_4$, 
 \item[(C2.3)] $\alpha_4 \leq \alpha_2+\alpha_3-1$, 
 \item[(C2.4)] $\alpha_{21}+\alpha_1\leq \alpha_4+\alpha_2-1$.
 \end{enumerate}
\end{lemma}

\begin{proof}
If tangent cone is Cohen-Macaulay then C$(2.1)$ and C$(2.2)$ comes from the indispensability of $f_3$  and $f_5$. If $\alpha_1>\alpha_{21}+2$, $f_4$ is indispensable, in which case C$(2.3)$ follows. If $\alpha_1=\alpha_{21}+2$, $f_4$ is not indispensable. To prove C$(2.3)$ in this case, assume contrary that  $\alpha_4 > \alpha_2+\alpha_3-1$. Then ${\rm LM}(f_4)=X_1X_2^{\alpha_2-1}X_3^{\alpha_3-1}$. As $f_4 \in I_S$, there exists a binomial $g$ in a minimal standard basis of $I_S$ such that  ${\rm LM}(g) \mid {\rm LM}(f_4)$ and as the tangent cone is Cohen-Macaulay $X_2 \nmid {\rm LM}(g)$. Hence ${\rm LM}(g)=X_1^{a}X_3^{b}$ 
 with $a\leq 1$ and $ b\leq \alpha_3-1$. Then ${\rm deg}(f_5)-{\rm deg}(g)=(\alpha_{21}+1-a)n_1+(\alpha_3-1-b)n_3\in S$ but this contradicts with the minimality of ${\rm deg}(f_5)$. Hence, C$(2.3)$ must hold.
 
 For the last condition,  the result follows immediately if $\alpha_4\geq \alpha_1$, as in this case, $\alpha_{21}+\alpha_1\leq \alpha_4+\alpha_{21}\leq \alpha_4+\alpha_2-1$. When $\alpha_4<\alpha_1$, assume contrary that $\alpha_{21}+\alpha_1> \alpha_4+\alpha_2-1$. Then, as $(\alpha_1+\alpha_{21})n_1=\alpha_2n_2+n_3+(\alpha_4-2)n_4$, the binomial $ f_6= X_1^{\alpha_1+\alpha_{21}}-X_2^{\alpha_2}X_3X_4^{\alpha_4-2} \in I_S$ and  ${\rm LM}(f_6)=X_2^{\alpha_2}X_3X_4^{\alpha_4-2}$ is divisible by $X_2$. As the tangent cone is Cohen-Macaulay there exists a nonzero polynomial $f$ in a minimal standard basis of $I_S$ such that ${\rm LM}(f)\mid {\rm LM}(f_6)$ and  $X_2 \nmid {\rm LM}(f)$. This implies that ${\rm LM}(f)=X_3^{a}X_4^{b}$, where $a \leq 1$ and $b\leq \alpha_4-2$, and that  ${\rm deg}(f_1)-{\rm deg}(f)=(1-a)n_3+(\alpha_4-1-b)n_4$ is also in $S$ which contradicts with the minimality of ${\rm deg}(f_1)$. Hence, C$(2.4)$ must hold.
\end{proof}
Before computing a standard basis, we observe the following.
\begin{remark}\label{remark1}  When $n_2$ is the smallest number in $\{n_1,n_2,n_3,n_4\}$, $\alpha_{21}+1\leq \alpha_{2}$ holds automatically. Indeed, as $f_2$ is $S$-homogeneous, $\alpha_{21}n_1<\alpha_{21}n_1+n_4=\alpha_2n_2 <\alpha_2n_1$ implying $\alpha_{21} < \alpha_{2}$.
\end{remark}
Now, we compute a standard basis under the conditions C$(2.1)$, C$(2.2)$, C$(2.3)$, and C$(2.4)$.

\begin{lemma} \label{std2}
Let $n_2$ be the smallest number in $\{n_1,n_2,n_3,n_4\}$ and

\begin{enumerate}

 \item[(C2.1)]  $\alpha_{21}+\alpha_3\leq \alpha_1$, 
 \item[(C2.2)] $\alpha_{21}+\alpha_3\leq \alpha_4$, 
 \item[(C2.3)] $\alpha_4 \leq \alpha_2+\alpha_3-1$, 
 \item[(C2.4)] $\alpha_{21}+\alpha_1\leq \alpha_4+\alpha_2-1$.

\end{enumerate}
then a minimal standard basis for $I_S$ is
\begin{enumerate}
\item[(i)] $\{f_1,f_2,f_3,f_4,f_5\}$ if $\alpha_1\leq \alpha_4$,
\item[(ii)] $\{f_1,f_2,f_3,f_4,f_5,f_6=X_1^{\alpha_1+\alpha_{21}}-X_2^{\alpha_2}X_3X_4^{\alpha_4-2}\}$ if $\alpha_1 > \alpha_4$, with respect 
\end{enumerate}to negative degree reverse lexicographical ordering with $X_3,X_4>X_1>X_2$.
\end{lemma}
\begin{proof} Omitted as it can be done similarly.
\end{proof}
We are now ready to give the full characterization.
\begin{theorem}
Suppose $n_2$ is the smallest number in $\{n_1,n_2,n_3,n_4\}$. Tangent cone of the monomial curve $C_S$ is Cohen-Macaulay if and only if
\begin{enumerate}
 \item[(C2.1)]  $\alpha_{21}+\alpha_3\leq \alpha_1$, 
 \item[(C2.2)] $\alpha_{21}+\alpha_3\leq \alpha_4$, 
 \item[(C2.3)] $\alpha_4 \leq \alpha_2+\alpha_3-1$, 
 \item[(C2.4)] $\alpha_{21}+\alpha_1\leq \alpha_4+\alpha_2-1$.
\end{enumerate}
\end{theorem}
\begin{proof}
If tangent cone is Cohen-Macaulay then C$(2.1)$, C$(2.2)$, C$(2.3)$ and C$(2.4)$ must hold by Lemma \ref{necessary2}. 
If C$(2.1)$, C$(2.2)$, C$(2.3)$ and C$(2.4)$ hold, then a minimal standard basis with respect to the negative degree reverse lexicographic ordering making $X_2$ the smallest variable is $G=\{f_1,f_2,f_3,f_4,f_5\}$ in the case $\alpha_4 \geq \alpha_1$ and $G=\{f_1,f_2,f_3,f_4,f_5,f_6\}$ in the case $\alpha_4<\alpha_1$ from Lemma \ref{std2}. $X_2$ does not divide ${\rm LM}(f_i)$ in both cases, so the tangent cone is Cohen-Macaulay by \cite[Lemma 2.7]{AMS}.
\end{proof}

\subsection{Cohen Macaulayness of the tangent cone when $n_3$ is smallest} In this section, we deal with the Cohen Macaulayness of the tangent cone when $n_3$ is the smallest number in $\{n_1,n_2,n_3,n_4\}$. As before, we get the necessary conditions first.
\begin{lemma} \label{necessary3} Suppose $n_3$ is the smallest number in $\{n_1,n_2,n_3,n_4\}$. If the tangent cone of the monomial curve $C_S$ is Cohen-Macaulay, then the following must hold
\begin{enumerate}
 \item[(C3.1)]  $\alpha_1 \leq \alpha_4$, 
 \item[(C3.2)]  $\alpha_4 \leq \alpha_{21}+\alpha_3$, 
 \item[(C3.3i)]  $\alpha_4 \leq \alpha_2+\alpha_3-1$ if $\alpha_1 - \alpha_{21} >2$,
  \item[(C3.3ii)]  $\alpha_4 \leq \alpha_2+2\alpha_3-3$ if $\alpha_1 - \alpha_{21} =2$, 

 \end{enumerate}
\end{lemma}

\begin{proof}
If tangent cone is Cohen-Macaulay then C$(3.1)$ and C$(3.2)$ comes from the indispensability of $f_1$  and $f_5$. If $\alpha_1>\alpha_{21}+2$, $f_4$ is indispensable, in which case C$(3.3i)$ follows. If $\alpha_1=\alpha_{21}+2$, $f_4$ is not indispensable. To prove C$(3.3ii)$ in this case, assume contrary that  $\alpha_4 > \alpha_2+2\alpha_3-3$. Then $\alpha_4 > \alpha_2+\alpha_3-1$ and ${\rm LM}(f_4)=X_1X_2^{\alpha_2-1}X_3^{\alpha_3-1}$. As ${\rm LM}(f_3)=X_1X_2 \mid {\rm LM}(f_4)$, $f_4$ can not be in a minimal standard basis of $I_S$. It can not be in a minimal generating set since a minimal generating set would lie in a minimal standard basis. Since Betti $S$-degrees are invariant, there must be a binomial of degree $d_4$ in a minimal generating set. We prove that $f'_4=X_4^{\al_4}-X_2^{\al_2-2}X_3^{2\al_3-1}$ must belong to a minimal generating set and so to a minimal standard basis. This will follow from \cite{CKT} and the claim that 
$$\deg_S^{-1}(d_4)=\{X_4^{\al_4}\} \cup \{X_1X_2^{\alpha_2-1}X_3^{\alpha_3-1},X_2^{\al_2-2}X_3^{2\al_3-1}\}.$$
In order to prove the claim above, take $m\in \deg_S^{-1}(d_4)$. Since $d_3$ is the only $S$-degree smaller than $d_4$ and $\deg_S^{-1}(d_3)=\{X_3^{\al_3},X_1X_2\}$, it follows that $X_3^{\al_3} \mid m$ or $X_1X_2 \mid m$ if $\deg_S(m)=d_4$. If $X_3^{\al_3} \mid m$, then  $m=X_3^{\al_3} m'$. If $m'\neq X_2^{\alpha_2-2}X_3^{\alpha_3-1}$, then $m' - X_2^{\alpha_2-2}X_3^{\alpha_3-1} \in I_S$, as this binomial is $S$-homogeneous of $S$-degree $d=d_4-d_3$. As $d_3$ is the only $S$-degree smaller than $d_4$, it follows that $d_3<_S d<_S d_4$. So, $2d_3<_S d_4$. On the other hand, by Lemma \ref{lemmaS}, we have
$$d_4-2d_3=n_1+(\al_2-1)n_2+(\al_3-1)n_3-n_1-n_2-\al_3n_3=(\al_2-2)n_2-n_3 \notin S.$$
Thus, $m' = X_2^{\alpha_2-2}X_3^{\alpha_3-1}$ and so $m=X_2^{\al_2-2}X_3^{2\al_3-1}$. By the same argument, if $X_1X_2 \mid m$ then $m=X_1X_2^{\alpha_2-1}X_3^{\alpha_3-1}$, hence the claim follows.

If $\alpha_4 > \alpha_2+2\alpha_3-3$, ${\rm LM}(f'_4)=X_2^{\alpha_2-2}X_3^{2\alpha_3-1}$ is divisible by $X_3$, contradicting to the Cohen-Macaulayness of the tangent cone. So, C$(3.3ii)$ follows.
\end{proof}

Before computing a standard basis, we observe the following.
\begin{remark}\label{remark3}  When $n_3$ is the smallest number, $\al_1-\alpha_{21} < \alpha_{3}$ holds automatically. Indeed, as $f_3$ is $S$-homogeneous, $(\al_1-\alpha_{21})n_3<(\al_1-\alpha_{21}-1)n_1+n_2=\alpha_3n_3 $.
\end{remark}
Now, we compute a standard basis.

\begin{lemma} \label{std3}
Let $n_3$ be the smallest number in $\{n_1,n_2,n_3,n_4\}$ and $\alpha_2 \leq \alpha_{21}+1$,
then a minimal standard basis for $I_S$ is
\begin{enumerate}
\item[(i)] $\{f_1,f_2,f_3,f_4,f_5,f_6\}$ if $C(3.1)$, $C(3.2)$ and $C(3.3i)$ hold
\item[(ii)] $\{f_1,f_2,f_3,f'_4,f_5,f_6\}$ when $C(3.1)$, $C(3.2)$ and $C(3.3ii)$ hold, with respect 
\end{enumerate} to negative degree reverse lexicographical ordering with $X_2>X_1,X_4>X_3$, where $f_6=X_1^{\alpha_1-1}X_4-X_2^{\alpha_2-1}X_3^{\alpha_3}$.
\end{lemma}

\begin{proof}Omitted as it can be done similarly.
\end{proof}

We are now ready to give a list of sufficient conditions.
\begin{corollary}
Let $n_3$ is the smallest number and and $\alpha_2 \leq \alpha_{21}+1$.
\begin{enumerate}
\item[(i)] If C$(3.1)$, C$(3.2)$ and C$(3.3i)$ hold, then the tangent cone of the monomial curve $C_S$ is Cohen-Macaulay. 
\item[(ii)] When C$(3.1)$, C$(3.2)$ and C$(3.3ii)$ hold, the tangent cone of the monomial curve $C_S$ is Cohen-Macaulay if and only if $\alpha_1 \leq \alpha_2+\alpha_3-1$.
\end{enumerate}
\end{corollary}
\begin{proof} (i) If C$(3.1)$, C$(3.2)$ and C$(3.3i)$ hold, then a minimal standard basis with respect to the negative degree reverse lexicographic ordering making $X_3$ the smallest variable is $G=\{f_1,f_2,f_3,f_4,f_5,f_6\}$  from Lemma \ref{std3}. $X_3$ does not divide ${\rm LM}(f_i)$, so the tangent cone is Cohen-Macaulay by \cite[Lemma 2.7]{AMS}. \\
(ii) When C$(3.1)$, C$(3.2)$ and C$(3.3ii)$ hold, then a minimal standard basis with respect to the negative degree reverse lexicographic ordering making $X_3$ the smallest variable is $G=\{f_1,f_2,f_3,f'_4,f_5,f_6\}$  from Lemma \ref{std3}. $X_3$ does not divide ${\rm LM}(f_i)$, for $i=1,\dots,5$, so the tangent cone is Cohen-Macaulay by \cite[Lemma 2.7]{AMS} if and only if $X_3$ does not divide ${\rm LM}(f_6)$ if and only if $\alpha_1 \leq \alpha_2+\alpha_3-1$.
\end{proof}

We finish the section by illustrating that $\alpha_2 \leq \alpha_{21}+1$ is not a necessary condition. 
\begin{example}
Let $(\alpha_{21},\alpha_1,\alpha_2,\alpha_3,\alpha_4)=(2,4,5,4,5)$. Then $(n_1,n_2,n_3,n_4)=(81,59,28,74)$.  SINGULAR computes a minimal standard basis for $I_S$ as $\{X_1X_2-X_3^5, X_1^2X_4-X_2^{4}, X_1^4-X_3X_4^4, X_2^5-X_1X_3^5X_4, X_2X_4^4-X_1^3X_3^4, X_4^5-X_1X_2^3X_3^4\}$ and thus  $I_{S^*}$ is generated by
 $G_*= \{ X_1X_2, X_1^2X_4, X_1^4, X_2^5, X_2X_4^4, X_4^5\}$. As $X_3$ does not divide these elements, the tangent cone is Cohen-Macaulay from \cite[Lemma 2.7]{AMS}.
\end{example}

\subsection{Cohen-Macaulayness of the tangent cone when $n_4$ is smallest} We get some necessary conditions first as before.
\begin{lemma}
Suppose $n_4$ is the smallest number in $\{n_1,n_2,n_3,n_4\}$. If the tangent cone of the monomial curve $C_S$ is Cohen-Macaulay then 
\begin{enumerate}
 \item[(C4.1)] $\alpha_1\leq \alpha_4$,
\item[(C4.2)] $\alpha_2 \leq \alpha_{21}+1$,
\item[(C4.3)] $\alpha_3+\alpha_{21}\leq \alpha_4$.
\end{enumerate}
\end{lemma}
\begin{proof}
The results follow immediately from the indispensabilities of $f_1$, $f_2$ and $f_5$ respectively.
\end{proof}
\begin{remark}\label{remark4}
If $n_4$ is the smallest number in $\{n_1,n_2,n_3,n_4\}$ then $\alpha_4>\alpha_{2}+\alpha_3-1$. Indeed, as $f_4$ is $S$-homogeneous and $n_4$ is the smallest number in $\{n_1,n_2,n_3,n_4\}$ $\alpha_4n_4=n_1+(\alpha_2-1)n_2+(\alpha_3-1)n_3>(\alpha_2+\alpha_3-1)n_4$ implying $\alpha_4>\alpha_{2}+\alpha_3-1$.
\end{remark}

\begin{lemma}\label{lemma4.2}
Let $n_4$ be the smallest number in $\{n_1,n_2,n_3,n_4\}$ and $ \alpha_3 \leq \alpha_1-\alpha_{21}$. If the conditions C$(4.1)$, C$(4.2)$ and C$(4.3)$ hold, then $\{f_1,f_2,f_3, f_4,f_5 \}$  is a minimal standard basis for $I_S$ with respect to negative degree reverse lexiographical ordering with $X_3>X_1,X_2>X_4$.
\end{lemma}
\begin{proof}Omitted as it can be done similarly.
\end{proof}
\begin{corollary} Let $n_4$ be the smallest number in $\{n_1,n_2,n_3,n_4\}$ and $ \alpha_3 \leq \alpha_1-\alpha_{21}$. If the conditions C$(4.1)$, C$(4.2)$ and C$(4.3)$ hold, then the tangent cone of $C_S$ is Cohen-Macaulay.  
\end{corollary}
\begin{proof}
By hypothesis $\{f_1,f_2,f_3, f_4,f_5 \}$  is a minimal standard basis for $I_S$ with respect to negative degree reverse lexiographical ordering with $X_4$ the smallest variable from Lemma \ref{lemma4.2} and $X_4\nmid {\rm LM}(f_i)$ for $i=1,2,3,4,5$. Thus, it follows from \cite[Lemma 2.7]{AMS} that the tangent cone is Cohen-Macaulay.
\end{proof}
However, the tangent cone may be Cohen-Macaulay even if $\alpha_3 > \alpha_1-\alpha_{21}$.
\begin{example}
Let $(\alpha_{21},\alpha_1,\alpha_2,\alpha_3,\alpha_4)=(4,7,3,4,9)$. Then $(n_1,n_2,n_3,n_4)=(97,154,87,74)$.  SINGULAR computes a minimal standard basis for $I_S$ as $\{X_1^{2}X_2-X_3^4, X_2^3-X_1^{4}X_4, X_1X_2^2X_3^3-X_4^9, X_2^2X_3^4-X_1^6X_4, X_1^7-X_3X_4^8, X_1^5X_3^3-X_2X_4^8, X_2X_3^7-X_1X_4^9, X_1^3X_3^7-X_2^2X_4^8, X_3^{11}-X_1^3X_4^9\}$ and so the ideal $I_{S^*}$ is generated by the set \newline
 $G_*= \{X_1^{2}X_2, X_2^3, X_1X_2^2X_3^3, X_2^2X_3^4, X_1^7, X_1^5X_3^3, X_2X_3^7, X_1^3X_3^7-X_2^2X_4^8,  X_3^{11}\}$. As $X_4$ does not divide elements, the tangent cone is Cohen-Macaulay from \cite[Lemma 2.7]{AMS}.
\end{example}

\section*{Acknowledgements} We would like to thank F. Arslan, A. Katsampekis and M. Morales for very helpful
suggestions. We also thank an anonymous referee for his/her comments that improved the presentation. Most of the examples are computed by using the computer algebra systems Macaulay $2$ and Singular, see \cite{Mac2,singular}.

\end{document}